\documentclass[a4paper,12pt,final]{amsart}
\usepackage{times,a4wide,mathrsfs,amssymb,amsmath,amsthm}

\newcommand{\C}{\mathbb{C}}
\newcommand{\ZZ}{\mathbb{Z}}

\newcommand{\QQ}{\mathbb{Q}}

\newcommand{\PP}{\mathbb{P}}

\newcommand{\OO}{\mathcal O}
\newcommand{\Ss}{\mathcal S}

\newcommand{\Sy}{\mathfrak S}

\newcommand{\Gr}{\hbox{Gr}}

\newcommand{\MM}{\mathcal M}

\newcommand{\rom}{\romannumeral}
\newcommand{\alb}{\hbox{Alb}}
\newcommand{\bir}{\hbox{Bir}}

\DeclareMathOperator{\aut}{Aut}

\DeclareMathOperator{\ima}{Im}

\newtheorem{theorem}{Theorem}[section]

\newtheorem{corollary}[theorem]{Corollary}
\newtheorem{proposition}[theorem]{Proposition}
\newtheorem{conjecture}[theorem]{Conjecture}
\newtheorem{remark}[theorem]{Remark}
\newtheorem{definition}[theorem]{Definition}
\newtheorem{convention}{Conventions}

\newtheorem{nonumbering}{Theorem}

\newtheorem{nonumberingc}{Corollary}

\newtheorem{nonumberingt}{Acknowledgements}

\begin{document}
\author[Robert Laterveer]
{Robert Laterveer}

\address{Institut de Recherche Math\'ematique Avanc\'ee,
CNRS -- Universit\'e 
de Strasbourg,\
7 Rue Ren\'e Des\-car\-tes, 67084 Strasbourg CEDEX,
FRANCE.}
\email{robert.laterveer@math.unistra.fr}

\title[Zero--cycles on self--products of surfaces]{Zero--cycles on self--products of surfaces: some new examples verifying Voisin's conjecture}

\begin{abstract} An old conjecture of Voisin describes how $0$--cycles of a surface $S$ should behave when pulled--back to the self--product $S^m$ for $m>p_g(S)$.
We exhibit some surfaces with large $p_g$ that verify Voisin's conjecture.
\end{abstract}


\maketitle

\section{Introduction}

Let $X$ be a smooth projective variety over $\C$, and let $A^i(X)_{\ZZ}:=CH^i(X)_{}$ denote the Chow groups of $X$ (i.e. the groups of codimension $i$ algebraic cycles on $X$ with $\ZZ$--coefficients, modulo rational equivalence \cite{F}). Let $A^i_{hom}(X)_{\ZZ}$ (and $A^i_{AJ}(X)_{\ZZ}$) denote the subgroup of homologically trivial (resp. Abel--Jacobi trivial) cycles.     

The Bloch--Beilinson--Murre conjectures present a beautiful and coherent dream--world in which Chow groups are determined by cohomology and the coniveau filtration \cite{J2}, \cite{J4}, \cite{Mur}, \cite{Kim}, \cite{MNP}, \cite{Vo}. The following particular instance of this dream--world was first formulated by Voisin:

\begin{conjecture}[Voisin 1993 \cite{V9}]\label{conj} Let $S$ be a smooth projective surface. Let $m$ be an integer larger than the geometric genus $p_g(S)$. Then for any $0$--cycles $a_1,\ldots,a_m\in A^2_{AJ}(S)_{\ZZ}$, one has
  \[ \sum_{\sigma\in\Sy_m} \hbox{sgn}(\sigma) a_{\sigma(1)}\times\cdots\times a_{\sigma(m)}=0\ \ \ \hbox{in}\ A^{2m}(S^m)_{\ZZ}\ .\]
  (Here $\Sy_m$ is the symmetric group on $m$ elements, and $ \hbox{sgn}(\sigma)$ is the sign of the permutation $\sigma$.)
  \end{conjecture}
  
For surfaces of geometric genus $0$, Conjecture \ref{conj} reduces to Bloch's conjecture \cite{B}. For surfaces $S$ of geometric genus $1$, Conjecture \ref{conj} takes on a particularly simple form: in this case, the conjecture stipulates that any $a_1, a_2\in A^2_{AJ}(S)_{\ZZ}$ should verify the equality
  \[ a_1\times a_2 =a_2\times a_1\ \ \ \hbox{in}\ A^4(S\times S)_{\ZZ}\ .\]
  This conjecture is still open for a general $K3$ surface;
examples of surfaces of geometric genus $1$ verifying this conjecture are given in \cite{V9}, \cite{16.5}, \cite{19}, \cite{21}. One can also formulate versions of Conjecture \ref{conj} for higher--dimensional varieties; this is studied in \cite{V9}, \cite{17}, \cite{24.4}, \cite{24.5}, \cite{BLP}, \cite{LV}, \cite{Ch}.

On a historical note, it is interesting to observe that Voisin's Conjecture \ref{conj} antedates Kimura's conjecture ``all varieties have finite--dimensional motive'' \cite{Kim}. Both conjectures have a similar flavour: Chow groups of a surface $S$ should have controlled behaviour when pulled--back to the self--product $S^m$, for large $m$. 
The difference between Voisin's conjecture and Kimura's conjecture lies in the index $m$ which is much lower in Voisin's conjecture. In fact (as explained in \cite{BLP}), Voisin's conjecture follows from a combination of Kimura's conjecture with a strong form of the generalized Hodge conjecture.

The goal of the present note is to collect some (easy) examples of surfaces with geometric genus larger than $1$ verifying Voisin's conjecture. 
 
 \begin{nonumbering}[=Corollaries \ref{main1}, \ref{cor2}, \ref{cor4} and \ref{last}] The following surfaces verify Conjecture \ref{conj}:
 
 \item
 {(\rom1)} generalized Burniat type surfaces in the family $\Ss_{16}$ of \cite{BCF} ($p_g(S)=3$);
 
  \item
 {(\rom2)} the hypersurfaces $S\subset A/\iota$ considered in \cite{LNP}, where $A$ is an abelian threefold and $\iota$ is the $-1$-involution ($p_g(S)=3$);
 
% \item{(\rom3)} the double cover of a cubic surface branched along the Hessian ($p_g(S)=4$);
 
  \item
 {(\rom3)} minimal surfaces $S$ of general type with $p_g(S)=q(S)=3$ and $K^2_S=6$;
 
 \item{(\rom4)} the double cover of certain cubic surfaces (among which the Fermat cubic)
 branched along the Hessian ($p_g(S)=4$); 
 
  \item
 {(\rom5)} the Fano surface of lines in a smooth cubic threefold ($p_g(S)=10$);
 
 \item{(\rom6)} the quotient $S=F/\iota$, where $F$ is the Fano surface of conics in a Verra threefold and $\iota$ is a certain involution ($p_g(S)=36$);
 
 \item{(\rom7)} the surface of bitangents $S$ of a general quartic in $\PP^3$ ($p_g(S)=45$);
 
 \item{(\rom8)} the singular locus $S$ of a general EPW sextic ($p_g(S)=45$).
 
 \end{nonumbering}

  A by--product of the proof is that these surfaces all have finite--dimensional motive, in the sense of Kimura \cite{Kim} (this appears to be a new observation for cases (\rom6)--(\rom8)). Also,
  certain instances of the generalized Hodge conjecture are verified:
 
 \begin{nonumberingc}[=Corollary \ref{ghc}] Let $S$ be any of the above surfaces, and let $m>p_g(S)$. Then the sub--Hodge structure
   \[ \wedge^m H^2(S,\QQ)\ \subset\ H^{2m}(S^m,\QQ) \]
   is supported on a divisor.
   \end{nonumberingc}

  The surfaces considered in this note have an interesting feature in common (which makes it easy to prove Conjecture \ref{conj} for them): for many of them, intersection product induces a surjection
    \[ A^1_{hom}(S)\otimes A^1_{hom}(S)\ \twoheadrightarrow\ A^2_{AJ}(S)\ .\]
    In the other cases (cases (\rom2), (\rom4), (\rom6)--(\rom8), which have $q(S)=0$), the surface $S$ is dominated by a surface $T$ with the property that the intersection product map
    \[ A^1_{hom}(T)\otimes A^1_{hom}(T)\ \to\ A^2_{AJ}(T)\ \]  
    surjects onto $\ima \bigl( A^2_{AJ}(S)\to A^2_{AJ}(T)\bigr)$. 
    
    Using this feature, to prove Conjecture \ref{conj} for these surfaces one is reduced to a problem concerning $0$--cycles on abelian varieties. This last problem has recently been solved by Vial \cite{Ch}, using a strong version of the generalized Hodge conjecture for generic abelian varieties.

   \vskip0.6cm

\begin{convention} In this note, the word {\sl variety\/} will refer to a reduced irreducible scheme of finite type over $\C$. A {\sl subvariety\/} is a (possibly reducible) reduced subscheme which is equidimensional. 

{\bf Unless indicated otherwise, all Chow groups will be with rational coefficients}: we will denote by $A_j(X)$ the Chow group of $j$--dimensional cycles on $X$ with $\QQ$--coefficients (and by $A_j(X)_{\ZZ}$ the Chow groups with $\ZZ$--coefficients); for $X$ smooth of dimension $n$ the notations $A_j(X)$ and $A^{n-j}(X)$ are used interchangeably. 

The notations $A^j_{hom}(X)$, $A^j_{AJ}(X)$ will be used to indicate the subgroups of homologically trivial, resp. Abel--Jacobi trivial cycles.
%For a morphism $f\colon X\to Y$, we will write $\Gamma_f\in A_\ast(X\times Y)$ for the graph of $f$.
The contravariant category of Chow motives (i.e., pure motives with respect to rational equivalence as in \cite{Sc}, \cite{MNP}) will be denoted $\MM_{\rm rat}$.

%To avoid heavy notation, if $\tau\colon Y\to X$ is a closed inclusion and $a\in A_iY$, we will frequently write $a\in A_iX$ to indicate the proper push--forward $\tau_\ast(a)$. Likewise, for any inclusion $Y\subset X$ and $b\in A^jX$ we will often write
%  \[   b\vert_{Y}\ \ \in A^jY\]
 % to indicate the cycle class $\tau^\ast(b)$.

%The Griffiths group $\grif^j$ is the group of codimension $j$ cycles that are homologically trivial modulo algebraic equivalence, again with $\QQ$--coefficients. 

We will write $H^j(X)$ 
to indicate singular cohomology $H^j(X,\QQ)$.
\end{convention}

  \section{Generalized Burniat type surfaces with $p_g=3$}
  
  \begin{definition}[\cite{BCF}]\label{gbt} Let $A=E_1\times E_2\times E_3$ be a product of elliptic curves. A {\em generalized Burniat type surface\/} (or ``GBT surface'')
  is a quotient $S=Y/G$, where $Y\subset A$ is a smooth hypersurface corresponding to the square of a principal polarization, and $G\cong \ZZ_2^3$ acts freely.
   \end{definition}
   
   \begin{remark} GBT surfaces are minimal surfaces of general type with $p_g(S)=q(S)$ ranging from $0$ to $3$. There are $16$ irreducible families of GBT surfaces, labelled $\Ss_1,\ldots \Ss_{16}$ in \cite{BCF}. The families $\Ss_1, \Ss_2$ have moduli--dimension $4$, the other families are $3$--dimensional.
    \end{remark}

  \begin{theorem}[Peters \cite{Chris}]\label{Gbt} Let $S$ be a GBT surface with $p_g(S)=3$ (i.e., $S$ is in the family labelled $\Ss_{16}$ in \cite{BCF}), and let 
  $A$ be the abelian threefold as in definition \ref{gbt}.
  
  \noindent
  (\rom1)
   $S$ has finite--dimensional motive, and there are natural isomorphisms
     \[ A^2_{(2)}(A)\ \xrightarrow{\cong}\ A^2_{AJ}(S)\ \xrightarrow{\cong}\ A^3_{(2)}(A)\ .\]
   (Here $A^\ast_{(\ast)}(A)$ refers to Beauville's decomposition \cite{Beau}.)  
   
  \noindent
  (\rom2) Intersection product induces a surjection
    \[ A^1_{hom}(S)\otimes A^1_{hom}(S)\ \twoheadrightarrow\ A^2_{AJ}(S)\ .\]   
       \end{theorem}
  
  \begin{proof} Part (\rom1) is \cite[Theorem 4.2]{Chris}.
  
  Part (\rom2) follows from (\rom1), in view of the fact that intersection product induces a surjection
  \[ A^1_{hom}(A)\otimes A^1_{hom}(A)\ \twoheadrightarrow\ A^2_{(2)}(A) \ \]
  \cite[Proposition 4]{Beau}.
    \end{proof}

  Property (\rom2) of Theorem \ref{Gbt} is relevant to Conjecture \ref{conj}:
  
  \begin{proposition}\label{handy0} Let $S$ be a smooth projective surface, and assume that intersection product induces a surjection
     \[ A^1_{hom}(S)\otimes A^1_{hom}(S)\ \twoheadrightarrow\ A^2_{AJ}(S)\ .\]
    Then $S$ has finite--dimensional motive.
    
Also, Conjecture \ref{conj} is true for $S$ with $m>{q(S)\choose 2}$.
 (In particular, in case of equality $p_g(S)= {q(S)\choose 2}$ the full Conjecture \ref{conj} is true for $S$.)       
  \end{proposition}
  
  \begin{proof} Let $\alpha\colon S\to A:=\alb(S)$ be the Albanese map. There is a commutative diagram
    \[ \begin{array}[c]{ccc}   
          A^1_{hom}(S)\otimes A^1_{hom}(S) &\to& A^2_{AJ}(S)\\
          &&\\
        \ \ \ \   \uparrow{\scriptstyle (\alpha^\ast,\alpha^\ast)}&&   \ \ \ \   \uparrow{\scriptstyle \alpha^\ast}\\
        &&\\
          A^1_{hom}(A)\otimes A^1_{hom}(A) &\to& A^2_{(2)}(A)\\  
         \end{array} \]
       (where horizontal maps are induced by intersection product, and $A^\ast_{(\ast)}(A)$ refers to the Beauville decomposition \cite{Beau} of the Chow ring of any abelian variety). As the left vertical map is an isomorphism, the assumption implies that the right vertical map is surjective. In view of \cite[Theorem 3.11]{V3}, this implies $S$ has finite--dimensional motive. (For an alternative proof of \cite[Theorem 3.11]{V3} in terms of birational motives, cf. \cite[Theorem B.7]{LNP}. For a similar result, cf. \cite[Proposition 2.1]{Diaz}.)  
       
    Next, let us consider Conjecture \ref{conj} for $S$. Thanks to Rojtman's result \cite{Ro}, it suffices to establish Conjecture \ref{conj} for $0$--cycles with $\QQ$--coefficients.
    Because $\alpha^\ast\colon A^2_{(2)}(A)\to A^2_{AJ}(S)$ is surjective, to prove Conjecture \ref{conj} for $S$ it suffices to prove (a version of) Conjecture \ref{conj} for elements $b_1,\ldots,b_m\in  A^2_{(2)}(A)$. We now reduce to $0$--cycles on $A$: given $b_j\in A^2_{(2)}(A)$, let
    \[ c_j:= b_j\cdot h^{q-2}\ \ \in\ A^q_{(2)}(A)\ ,\ \ \ j=1,\ldots,m\ ,\]
    be $0$--cycles, where $q:=q(S)$ is the dimension of $A$ and $h\in A^1(A)$ is a symmetric ample divisor.  
    Let us consider the $\Sy_m$--invariant ample divisor
     \[ H:= \sum_{j=1}^m (pr_j)^\ast(h)\ \ \ \in\ A^1(A^m)\ .\]
     From K\"unnemann's hard Lefschetz result \cite{Kun}, we know that the map
     \[ \cdot H^{m(q-2)}\colon\ \ A^{2m}_{(2m)}(A^m)\ \to\ A^{qm}_{(2m)}(A^m) \]
     is an isomorphism. On the other hand,
     \[ \begin{split}
     c_{\sigma(1)}\times\cdots\times c_{\sigma(m)}&= \bigl(b_{\sigma(1)}\times\cdots\times b_{\sigma(m)} \bigr)\cdot \bigl( h^{q-2}\times\cdots\times h^{q-2}\bigr)\\
                                                                              &=    \bigl(b_{\sigma(1)}\times\cdots\times b_{\sigma(m)} \bigr)\cdot H^{m(q-2)}\ \ \ \hbox{in}\  A^{qm}_{(2m)}(A^m)\\
                                                                         \end{split}\]
                                                (since intersecting $A^2(A)$ with a power $h^r, r>q-2$ gives $0$).
                                                
          We are thus reduced to proving that for any $c_1,\ldots,c_m\in A^q_{(2)}(A)$, where $m>{q\choose 2}$, there is equality
          \[ \sum_{\sigma\in\Sy_m} \hbox{sgn}(\sigma) \,  c_{\sigma(1)}\times\cdots\times c_{\sigma(m)}=0\ \ \ \hbox{in}\ A^{gm}(A^m)_{}\ .\]                                                                
       At this point, we can invoke the following general result on $0$--cycles on abelian varieties to conclude:
    
    \begin{theorem}[Vial \cite{Ch}] Let $A$ be an abelian variety of dimension $g$, and let $c_1,\ldots,c_m\in A^g_{(k)}(A)$.
    
    If $k$ is even and $m>{g\choose k}$, there is vanishing
      \[  \sum_{\sigma\in\Sy_m} \hbox{sgn}(\sigma) \, c_{\sigma(1)}\times\cdots\times c_{\sigma(m)}=0\ \ \ \hbox{in}\ A^{mg}(A^m)_{}\ .\]   
      
 If $k$ is odd and $m>{g\choose k}$, there is vanishing
      \[  \sum_{\sigma\in\Sy_m}  c_{\sigma(1)}\times\cdots\times c_{\sigma(m)}=0\ \ \ \hbox{in}\ A^{mg}(A^m)_{}\ .\]   
      \end{theorem}
      
     \begin{proof} This is \cite[Theorem 4.1]{Ch}, whose proof uses the concept of ``generically defined cycles on abelian varieties'', and a strong form of the generalized Hodge conjecture for powers of generic abelian varieties, due to Hazama \cite[Theorem 2.12]{Ch}. The case $k=g$ was proven earlier (and differently) in 
     \cite[Example 4.40]{Vo}.
     \end{proof} 
    
   This ends the proof of Proposition \ref{handy0}.                             
  \end{proof}

  We can now prove that surfaces in the family $\Ss_{16}$ verify Voisin's conjecture:
  
  \begin{corollary}\label{main1} Let $S$ be a GBT surface with $p_g(S)=3$ (i.e., $S$ is in the family labelled $\Ss_{16}$ in \cite{BCF}).  
  Then $S$ verifies Conjecture \ref{conj}: for any $m>3$ and $a_1,\ldots,a_m\in A^2_{AJ}(S)$, there is equality
    \[  \sum_{\sigma\in\Sy_m} \hbox{sgn}(\sigma) a_{\sigma(1)}\times\cdots\times a_{\sigma(m)}=0\ \ \ \hbox{in}\ A^{2m}(S^m)\ .\]  
    \end{corollary}
    
  \begin{proof} This follows from Proposition \ref{handy0}, in view of Theorem \ref{Gbt} plus the fact that $q(S)=p_g(S)=3$.
    \end{proof}

 We recall that the truth of Conjecture \ref{conj} implies a certain instance of the generalized Hodge conjecture:
 
 \begin{corollary}\label{ghc} Let $S$ be a surface verifying Conjecture \ref{conj}, and let $m>p_g(S)$. Then the sub--Hodge structure
    \[ \wedge^m H^2(S,\QQ)\ \subset\ H^{2m}(S^m,\QQ) \]
   is supported on a divisor.
   \end{corollary} 
  
\begin{proof} This is already observed in \cite{V9}.  Consider the Chow motive $\wedge^m h^2(S)$ defined by the idempotent
  \[  \Gamma:=  \bigl(\sum_{\sigma\in\Sy_m} \hbox{sgn}(\sigma) \Gamma_\sigma\bigr)\circ \bigl(\pi^2_S\times\cdots\times \pi^2_S\bigr)\ \ \ \in\ A^{2m}(S^m\times S^m)\ .\]
  Conjecture \ref{conj} is equivalent to saying that $A_0(\wedge^m h^2(S))=0$.
  
  Applying the Bloch--Srinivas argument \cite{BS} to $\Gamma$, one obtains a rational equivalence
  \[ \Gamma=\gamma\ \ \ \hbox{in}\ A^{2m}(S^m\times S^m)\ ,\]
  where $\gamma$ is a cycle supported on $S^m\times D$ for some divisor $D\subset S^m$.
  On the other hand, $\Gamma$ acts on $H^{2m}(S^m,\QQ)$ as projector on $\wedge^m H^2(S,\QQ)$. It follows that $ \wedge^m H^2(S,\QQ)$ is supported on $D$.  
  \end{proof}

 \section{A criterion}

 The approach of the last section can be conveniently rephrased as follows:
 
 \begin{proposition}\label{handy} Let $S$ be a smooth projective surface. Assume that $S$ has finite--dimensional motive, and that cup product induces an isomorphism
   \[  C\colon\ \    \wedge^2 H^1(S,\OO_S) \ \xrightarrow{\cong}\ H^2(S,\OO_S)\ .\]
   Then Conjecture \ref{conj} is true for $S$.
   \end{proposition}
   
  \begin{proof} Surjectivity of $C$, combined with finite--dimensionality of the motive of $S$, ensures that intersection product induces a surjection
    \[ A^1_{hom}(S)\otimes A^1_{hom}(S)\ \twoheadrightarrow\ A^2_{AJ}(S)\ \]
    \cite{moib}. The assumption that $C$ is an isomorphism implies that $p_g(S)={{q(S)}\choose{2}}$. The result now follows from Proposition \ref{handy0}.
   \end{proof} 
    
  This takes care of two more cases announced in the introduction:
  
  \begin{corollary}\label{cor2} Conjecture \ref{conj} is true for the following surfaces:
  
  \item
 {(\rom1)} minimal surfaces of general type with $p_g(S)=q(S)=3$ and $K^2=6$;
 
  \item
 {(\rom2)} the Fano surface of lines in a cubic threefold ($p_g(S)=10$).
 \end{corollary}
 
 \begin{proof} 
 In case (\rom1), it is known that $S$ is the symmetric square $S=C^{(2)}$ where $C$ is a genus $3$ curve \cite{CCML} (cf. also \cite[Theorem 9]{BCP}). Thus, the assumptions of Proposition \ref{handy} are clearly satisfied.
 
 As for case (\rom2), it is well--known this satisfies the assumptions of Proposition \ref{handy} (finite--dimensionality is proven in \cite{Diaz} and \cite{22}). Alternatively, one could apply Proposition \ref{handy0} directly (the assumption of Proposition \ref{handy0} is satisfied by the Fano surface thanks to \cite{B}; an alternative proof is sketched in \cite[Remark 20.8]{SV}).
 \end{proof}

\section{A variant criterion}

Let us now state a variant version of Proposition \ref{handy0}:
  
  \begin{proposition}\label{handy1} Let $S$ be a smooth projective surface. Assume that $S=S^\prime/<\iota>$, where $\iota$ is an automorphism of a surface $S^\prime$ 
  such that intersection product induces a surjection
   \[     A^1_{hom}(S^\prime)\otimes   A^1_{hom}(S^\prime) \ \twoheadrightarrow\ A^2_{AJ}(S^\prime)^\iota\ .\]
   Then $S$ has finite--dimensional motive.
   
   Also, Conjecture \ref{conj} is true for $S$ with $m>{q(S^\prime)\choose 2}$.
   (In particular, if $p_g(S)={q(S^\prime)\choose 2}$ the full Conjecture \ref{conj} is true for $S$.)
      \end{proposition}
      
  \begin{proof} This is proven just as Proposition \ref{handy0}.
    \end{proof}

  This takes care of several more cases announced in the introduction:
  
 \begin{corollary}\label{cor4} Conjecture \ref{conj} is true for the following surfaces:
  
 \item
 {(\rom1)} surfaces $S=T/<\iota>$, where $T$ is a smooth divisor in the linear system $\vert 2\Theta\vert$ on a principally polarized abelian threefold, and $\iota$ is the $(-1)$--involution ($p_g(S)=3$);

 \item{(\rom2)} the quotient $S=F/\iota$, where $F$ is the Fano surface of conics in a general Verra threefold and $\iota$ is a certain involution ($p_g(S)=36$);
 
 \item{(\rom3)} the surface of bitangents $S$ of a general quartic in $\PP^3$ ($p_g(S)=45$);
 
 \item{(\rom4)} the surface $S$ that is the singular locus of a general EPW sextic ($p_g(S)=45$).  
  \end{corollary}
  
  \begin{proof}
  
  \noindent
  \item{(\rom1)} The surface $S$ verifies the assumptions of Proposition \ref{handy1} with $S^\prime=T$, according to \cite[Subsection 7.2]{LNP}.
  
  \noindent
  \item{(\rom3)} More generally, one may consider the surface $S$ studied by Welters \cite{Wel} and defined as follows. Let $Y$ be a {\em quartic double solid\/}, i.e.
  $Y\to\PP^3$ is a double cover branched along a smooth quartic $Q$. Let $T$ be the surface of conics contained in $Y$, and let $\iota\in\aut(T)$ be the involution induced by the covering involution of $Y$. 
 Then the surface $S:=T/<\iota>$ is a smooth surface of general type with $p_g(S)=45$. 
 
 (The generic quartic $K3$ surface $Q$ does not contain a line. In this case, as explained in \cite{Fer} (cf. also \cite[Example 3.5]{Beau1} and \cite[Remark 8.5]{Huy1}), the surface $S$ is (isomorphic to) the so--called ``surface of bitangents'', which is the fixed locus of Beauville's anti--symplectic involution
    \[ Q^{[2]}\ \to\ Q^{[2]} \]
   first considered in \cite{Beau0}. As noted in \cite[Example 3.5]{Beau1}, this gives another proof of the fact that $p_g(S)=45$.)
   
   Voisin has proven \cite[Corollaire 3.2(b)]{V8} (cf. also \cite[Remarque 3.4]{V8}) that intersection product induces a surjection 
   \[ A^1_{hom}(T)\otimes A^1_{hom}(T)\ \twoheadrightarrow\ A^2_{AJ}(T)^\iota=A^2_{AJ}(S)\ .\]
   Since $p_g(S)=45$ and $q(T)=10$ \cite{Wel}, the assumptions of Proposition \ref{handy1} are met with.
   
 \noindent
 \item
 {(\rom2)} A {\em Verra threefold\/} $Y$ is a divisor of bidegree $(2,2)$ in $\PP^2\times\PP^2$ (these varieties were introduced in \cite{Ver}). Let $F$ be the Fano surface of conics of bidegree $(1,1)$ contained in $Y$. As observed in \cite[Section 5]{IKKR}, $F$ admits an involution $\iota$ such that $(F,\iota)$ enters into the set--up of Voisin's work \cite{V8}. Thus, \cite[Corollaire 3.2(b)]{V8} implies that intersection product induces a surjection
   \[   A^1_{hom}(F)\otimes A^1_{hom}(F)\ \twoheadrightarrow\ A^2_{AJ}(F)^\iota=A^2_{AJ}(S)\ .\]  
   Since $q(F)=9$ and $p_g(S)=36$ \cite[Proposition 5.1]{IKKR}, the assumptions of Proposition \ref{handy1} are again met with.
   
  \noindent
  \item
  {(\rom4)} Let $Y$ be a transverse intersection of the Grassmannian $Gr(2,5)\subset\PP^9$ with a codimension $2$ linear subspace and a quadric (i.e., $Y$ is an {\em ordinary Gushel--Mukai threefold\/}, in the language of \cite{DK}, \cite{DK1}). For generic $Y$, the surface $F$ of conics contained in $Y$ is smooth and irreducible.
  There exists a birational involution $\iota\in\bir(F)$, such that intersection product induces a surjection
    \[  A^1_{hom}(F)\otimes A^1_{hom}(F)\ \twoheadrightarrow\ A^2_{AJ}(F)^\iota\ \]  
  \cite[Corollaire 3.2(b)]{V8}. The surface $F$ and the birational involution $\iota$ are also studied in \cite{Lo} and \cite{DIM}. There exists a (geometrically meaningful) birational morphism $F\to F_m$, where $F_m$ is smooth and such that $\iota$ extends to a morphism $\iota_m$ on $F_m$ \cite{Lo}, \cite[Section 6]{DIM}, \cite[Section 5.1]{IM}. For $Y$ generic, the quotient $S:=F_m/<\iota_m>$ is smooth, and it is isomorphic to the singular locus of the EPW sextic associated to $Y$. 
 (This is contained in \cite{Lo}, \cite{DIM}. The double cover $F_m\to S$ is also described in \cite[Theorem 5.2(2)]{DK3}.) 
  
  Since $A^1_{hom}(), A^2_{AJ}()$ are birational invariants among smooth varieties, Voisin's result implies there is also a surjection
     \[  A^1_{hom}(F_m)\otimes A^1_{hom}(F_m)\ \twoheadrightarrow\ A^2_{AJ}(F_m)^{\iota_m}=A^2_{AJ}(S)\ .\]   
   It is known that $q(F_m)=10$ \cite{Lo} and $p_g(S)=45$ \cite{OG0} (this can also be deduced from \cite{Beau1}), and so Proposition \ref{handy1} applies.  
      \end{proof}

 \begin{remark} In cases (\rom2), (\rom3) and (\rom4) of Corollary \ref{cor4}, the surface $S$ is the fixed locus of an anti--symplectic involution of a hyperk\"ahler fourfold. For the surface of bitangents, this is Beauville's involution on the Hilbert square $Q^{[2]}$. 
 For the singular locus $S$ of a general EPW sextic, this is (isomorphic to) the fixed locus of the anti--symplectic involution of the associated double EPW sextic.
 
 For the surface $S$ of (\rom2), this is the anti--symplectic involution of the ``double EPW quartic'' (double EPW quartics form a $19$--dimensional family of hyperk\"ahler fourfolds, introduced in \cite{IKKR}).
 
 Is this merely a coincidence, or is there something fundamental going on ? Do other two--dimensional fixed loci of anti--symplectic involutions of hyperk\"ahler fourfolds 
 also enter in the set--up of Proposition \ref{handy1} ?
 \end{remark}

 \begin{remark} Inspired by the famous results concerning the Fano surface of the cubic threefold, Voisin \cite{V8} systematically studies the Fano surface $F$ of conics contained in Fano threefolds $Y$. Under certain conditions, she is able to prove \cite[Corollaire 3.2]{V8} that there is a birational involution $\iota$ on $F$, with the property that
   \[ A^1_{hom}(F)\otimes A^1_{hom}(F)\ \to\ A^2_{AJ}(F)^{<\iota>} \]
 is surjective (and so one could hope to apply Proposition \ref{handy1} to find more examples of surfaces verifying Conjecture \ref{conj}).  
 
 Examples given in \cite{V8} (other than those mentioned in Corollary \ref{cor4} above) include: 
 
 \noindent
 \item{(1)}
  Fano threefolds $Y$ of index $1$, Picard number $1$ and genus $g\in[7,10]\cup\{12\}$ \cite[Section 2.4]{V8};

  \noindent
  \item{(2)}
  a general complete intersection of two quadrics in $\PP^5$ \cite[Section 2.7]{V8};

  \noindent
  \item{(3)}
  the intersection of the Grassmannian $Gr(2,5)\subset\PP^9$ with a general codimension $3$ linear subspace \cite[Section 2.7]{V8}.
    
   (In all these cases, $\iota$ is actually the identity.) 
   
   In case (1), the surface of conics $F$ is not very interesting. (for $g=12$, $F\cong\PP^2$ \cite[Proposition B.4.1]{KPS}; for $g=10$, $F$ is an abelian surface \cite[Proposition B.5.5]{KPS}; ; for $g=9$, $F$ is a $\PP^1$--bundle over a curve \cite[Proposition 2.3.6]{KPS}; for $g=8$, $F$ is isomorphic to the Fano surface of a cubic threefold \cite[Proposition B.6.1]{KPS}; for $g=7$, $F$ is the symmetric product of a curve of genus $7$ \cite{Kuz05}. These results are also discussed in \cite[Section 3.1]{IM0}.)
   
   The other two cases also turn out to reduce to known cases: Indeed, for case (2) the Fano surface of lines is isomorphic to the Jacobian of a genus $2$ curve \cite[Theorem 2]{DR}. For case (3), the Fano threefold $Y$ is birational to a cubic threefold $Y^\prime$, and the Fano surface of conics on $Y$ is birational to the Fano surface of lines on 
   $Y^\prime$ \cite[Theorem B and Section 6]{Puts}. Since Conjecture \ref{conj} is obviously a birationally invariant statement, Conjecture \ref{conj} for the Fano surface of case (3) thus reduces to Corollary \ref{cor2}(\rom2).
    \end{remark}

 \begin{remark} There are interesting relations between the surfaces of Corollary \ref{cor4} and other Fano surfaces:
 
 In case (\rom2), the general Verra threefold $Y$ is birational to a one--nodal ordinary Gushel--Mukai threefold $\bar{X}$, and there is an induced birational map between the Fano surface of lines $F(Y)$ and the Fano surface of conics $F(\bar{X})$ \cite[Section 5.4 and Proposition 6.6]{DIM2}. 
 
 In case (\rom3), the general quartic double solid $Y$ is known to be birational to a one--nodal ordinary degree $10$ Fano threefold $\bar{X}$, and there is an induced birational map between the Fano surface of lines $F(Y)$ and the Fano surface of conics $F(\bar{X})$ \cite[Proposition 5.2]{DIM}. 
 %Voisin's work \cite{V8} thus implies that 
  \end{remark}

 \section{Double covers of cubic surfaces}
 
 \begin{theorem}[Ikeda \cite{Ike}]\label{ike} Let $Y\subset\PP^3$ be a smooth cubic surface, and let $\bar{S}\to Y$ be the double cover of $Y$ branched along its Hessian. Let $S\to\bar{S}$ be a minimal resolution of singularities. The surface $S$ is a minimal surface of general type with $p_g(S)=4$ and $K^2=6$.
 \end{theorem} 
 
 \begin{remark} The intersection of $Y$ with its Hessian is smooth (and so $S=\bar{S}$) precisely when $Y$ has no Eckardt points. In this case, the Picard rank of $S$ is $28$ \cite[Theorem 6.1]{Ike}. At the other extreme, if $Y$ is the Fermat cubic (which is the only cubic surface attaining the maximal number of Eckardt points) the Picard rank of $S$ is $44$ \cite[Theorem 6.6]{Ike}, and so in this case $S$ is a $\rho$--maximal surface (in the sense of \cite{Beau3}). For more on Eckardt points of cubic surfaces, cf. \cite[Chapter 2 Section 3.6]{Huy}.
 
% As follows from \cite[Proposition 4(a)]{Beau3} (combined with \cite{Ike}), the pencil of cubic surfaces
%   \[ x_0^3 + x_1^3 +x_2^3 -3\lambda x_0 x_1 x_2 + x_3^3 =0 \]
% contains a countably infinite number of surfaces $Y_\lambda$ (among which the Fermat cubic $Y_0$) for which $S$ is $\rho$--maximal. 
    \end{remark}

  Let us now prove Voisin's conjecture for some of Ikeda's double covers:

   \begin{corollary}\label{last} Let $Y\subset\PP^3$ be a smooth cubic surface, and let $S$ be a double cover as in theorem \ref{ike}. 
   Assume that $Y$ is in the pencil
      \[ x_0^3 + x_1^3 +x_2^3 -3\lambda x_0 x_1 x_2 + x_3^3 =0 \ .\]   
       Then $S$ verifies Conjecture \ref{conj}: for any $m>4$ and $a_1,\ldots,a_m\in A^2_{hom}(S)_{\ZZ}$, there is equality
    \[  \sum_{\sigma\in\Sy_m} \hbox{sgn}(\sigma) a_{\sigma(1)}\times\cdots\times a_{\sigma(m)}=0\ \ \ \hbox{in}\ A^{2m}(S^m)_{\ZZ}\ .\]  
    \end{corollary}
    
   \begin{proof} A first part of the argument works for arbitrary smooth cubic surfaces $Y$; only in the last step will we use that $Y$ is of a specific type.
   Let us assume $Y\subset\PP^3$ is any smooth cubic, defined by a cubic polynomial $f(x_0,\ldots,x_3)$. Let $Z\subset\PP^4$ be the smooth cubic threefold defined by
     \[ f(x_0,\ldots,x_3)+x_4^3=0\ ,\]
     so $Z$ has the structure of a triple cover
     \[ \rho\colon\ \ Z\ \to\ \PP^3 \]
     branched along $Y$.
     Let $F(Z)$ denote the Fano surface of lines contained in $Z$. Ikeda \cite{Ike} shows that there is a dominant rational map of degree $3$
     \[ f\colon\ \ F(Z)\ \dashrightarrow\ S \ ,\]
     and an isomorphism
     \[ f^\ast\colon\ \ H^2_{tr}(S,\QQ)\ \xrightarrow{\cong}\ H^2_{tr}(F(Z),\QQ)^{Gal(\rho)}\ .\]
     This implies that there is an isomorphism of homological motives
     \begin{equation}\label{homiso} {}^t \Gamma_f\colon\ \ \ t(S)\ \xrightarrow{\cong}\ t(F(Z))^{Gal(\rho)}:=(F(Z),{1\over 3}\sum_{g\in Gal(\rho)} \Gamma_g\circ \pi^2_{tr},0)\ \ \ \hbox{in}\ \MM_{\rm hom}\ .\end{equation}
   (Here for any surface $T$, the motive $t(T):=(T,\pi^2_{tr},0)\in\MM_{\rm rat}$ denotes the {\em transcendental part of the motive\/} as in \cite{KMP}.)
   
   According to \cite{Diaz} and \cite{22}, the Fano surface $F(Z)$ has finite--dimensional motive (in the sense of Kimura \cite{Kim}, \cite{An}, \cite{J4}). The surface $S$, being rationally dominated by $F(Z)$, also has finite--dimensional motive. Thus, one may upgrade (\ref{homiso}) to an isomorphism of Chow motives
    \[ {}^t \Gamma_f\colon\ \ \ t(S)\ \xrightarrow{\cong}\ t(F(Z))^{Gal(\rho)}\ \ \ \hbox{in}\ \MM_{\rm rat}\ .\]   
    In particular, this implies that there is an isomorphism of Chow groups
    \[ f^\ast\colon  A^2_{hom}(S)=A^2_{AJ}(S)\ \xrightarrow{\cong}\ A^2_{AJ}(F(Z))^{Gal(\rho)}\ .\]
%    What's more, there is a splitting of Chow motives
%    \[ t(F(Z))=M\oplus N\ \ \ \hbox{in}\ \MM_{\rm rat}\ ,\]
%    where we define
%     \[  \begin{split}  
%      &M:=t(F(Z))^{Gal(\rho)}\ ,\\
%      &N:=
%       (F(Z), \pi^2_{tr}-  {1\over 3}\sum_{g\in Gal(\rho)} \Gamma_g\circ \pi^2_{tr},0)   \ \ \ \hbox{in}\ \MM_{\rm rat}\ .\\
%       \end{split}\]
%   This splitting of Chow motives induces a splitting of Chow groups
%    \[ A^2_{AJ}(F(Z))= A^2(M)\oplus A^2(N) = f^\ast A^2_{hom}(S)\oplus A^2(N)\ .\]
%    We know from theorem \ref{main1} that for any $m^\prime>10$ one has
%    \[  \wedge^{m^\prime} A^2_{AJ}(F(Z))=  A_0\bigl(\wedge^{m^\prime} t(F(Z))\bigr)=    0\ .\]    
%    Using the above splitting, this can be rewritten as
%    \[  \bigoplus_{m+n=m^\prime}  \wedge^m A^2(M)\otimes \wedge^n A^2(N) = 0\ \ \ \forall\ m^\prime>10\ .\] 
%    Assuming now that $m>4$, we have in particular that
%    \[   \wedge^m A^2(M)\otimes \wedge^6 A^2(N) = 0\ .\] 
%    
       
   Let $A$ be the 5--dimensional Albanese variety of $F(Z)$ (which is isomorphic to the intermediate Jacobian of $Z$). As observed in \cite{Diaz}, the inclusion $F(Z)\hookrightarrow A$ induces an isomorphism
     \[ A^2_{(2)}(A)\cong A^2_{AJ}(F(Z))\ .\]
     In particular, there is a restriction--induced isomorphism
     \[    A^2_{(2)}(A)^{Gal(\rho)}\cong A^2_{AJ}(F(Z))^{Gal(\rho)}\ ,\]       
     where we simply use the same letter $\rho$ for the action induced by the triple cover $\rho\colon Z\to\PP^3$.
     
    Consequently, it suffices to prove a version of Conjecture \ref{conj} for cycles in $ A^2_{(2)}(A)^{Gal(\rho)}$. Also, using K\"unnemann's hard Lefschetz theorem (for some $Gal(\rho)$--invariant ample divisor), one reduces to a statement for cycles in $ A^5_{(2)}(A)^{Gal(\rho)}$ (i.e., $0$--cycles). This last statement can be proven, subject to some restrictions on the cubic surface $Y$, thanks to the following result:  
    
  \begin{proposition}[Vial \cite{Ch}]\label{factors} Let $B$ be an abelian variety of dimension $g$, and assume $B$ is isogenous to $ E_1^{r_1}\times E_2^{r_2}\times E_3^{r_3}$, where the $E_j$ are elliptic curves. Let $\Gamma\in A^g(B\times B)$ be an idempotent which lies in the sub--algebra generated by symmetric divisors. Assume that $\Gamma^\ast H^{j,0}(B)=0$ for all $j$. Then also 
    \[ \Gamma_\ast A^g(B)=0\ .\]
   \end{proposition}    
      
    \begin{proof} This is a special case of \cite[Theorem 3.15]{Ch}, whose hypotheses are more general.
        \end{proof}  
      
   It remains to verify that Proposition \ref{factors} applies to our set--up. If the cubic threefold $Z=Z_\lambda$ is in the pencil
     \[ x_0^3 + x_1^3 +x_2^3 -3\lambda x_0 x_1 x_2 + x_3^3 +x_4^4=0 \ ,\]
    its intermediate Jacobian $A$ is isogenous to $E_0^3\times E_\lambda^2$, where $E_\lambda$ is the elliptic curve
    \[ x_0^3 + x_1^3 +x_2^3 -3\lambda x_0 x_1 x_2=0\ \]
    \cite{Rou}. 
    We can apply Proposition \ref{factors} with $B:=A^m$ and
      \[ \Gamma:=  \bigl(\sum_{g\in Gal(\rho)} \Gamma_g\times \cdots \times\Gamma_g\bigr)  \circ   \bigl(\sum_{\sigma\in \Sy_m}  \hbox{sgn}(\sigma)\,\Gamma_\sigma\bigr) \circ \bigl(  \pi^8_A\times \cdots \times \pi^8_A\bigr) \ \ \ \in A^{5m}(A^m\times A^m)\ .\]
      Here $\pi^8_A$ is part of the Chow--K\"unneth decomposition of \cite{DM}, with the property that 
        \[ A^5_{(2)}(A)=(\pi^8_A)_\ast A^5(A)\ .\]
    Since $g\in Gal(\rho)$ and $\sigma\in \Sy_m$ are homomorphisms of abelian varieties, and the $\pi^8_A$ are symmetrically distinguished (in the sense of O'Sullivan \cite{OS}) and generically defined (in the sense of Vial \cite{Ch}), the correspondence $\Gamma$ is in the sub--algebra generated by symmetric divisors \cite[Proposition 3.11]{Ch}. In particular, the correspondence $\Gamma$ is symmetrically distinguished, and so (since it is idempotent in cohomology) idempotent. 
        
    The correspondence ${}^t \Gamma$ acts on cohomology as projector on 
      \[ \wedge^m \bigl( H^2(A)^{Gal(\rho)}\bigr)\ .\]
      Since 
      \[ \dim \Gr^0_F H^2(A)^{Gal(\rho)}=p_g(S)=4\ ,\]
   we have that $\Gamma^\ast=({}^t \Gamma)_\ast$ is zero on $H^{j,0}(B)$ as soon as $m>4$. Applying Proposition \ref{factors}, we can prove Conjecture \ref{conj} for 
   $A^5_{(2)}(A)^{Gal(\rho)}$ (and hence, as explained above, also for $A^2_{AJ}(S)$): let $b_1,\ldots,b_m\in A^5_{(2)}(A)^{Gal(\rho)}$, where $m>4$. Then
   \[ \sum_{\sigma\in\Sy_m} \hbox{sgn}(\sigma)\, b_{\sigma(1)}\times b_{\sigma(2)}\times\cdots\times b_{\sigma(m)}=\Gamma_\ast (b_1\times b_2\times\cdots\times b_m)=0\ \ \ \hbox{in}\ A^{5m}(A^m)\ .\]
            \end{proof}

  \begin{remark} The argument of Corollary \ref{last} also applies to double covers of some other cubic surfaces. For instance, let $Y$ be a cubic surface, let $S$ be the double cover as in theorem \ref{ike}, and let $J(Z)$ be the intermediate Jacobian of the associated cubic threefold. If $J(Z)$ is $\rho$--maximal, then $S$ verifies conjecture \ref{conj}. Indeed, $\rho$--maximality implies that $J(Z)$ is isogenous to $E^5$ for some elliptic curve $E$ \cite[Proposition 3]{Beau3}, and so Proposition \ref{factors} applies.
   \end{remark}

%  \section{Welters surfaces}
%  
% \begin{theorem}[Welters \cite{Wel}, Voisin \cite{V8}]\label{wel} Let $Y\to\PP^3$ be a double cover branched along a smooth quartic $Q$. Let $T$ be the surface of conics contained in $Y$, and let $\iota\in\aut(T)$ be the involution induced by the covering involution of $Y$. 
%  
%  \noindent
%  (\rom1) The surface $S:=T/<\iota>$ is a smooth surface of general type with $p_g(S)=45$.
%  
% 
% 
%  \noindent
%  (\rom2) Intersection product induces a surjection
%    \[ A^1_{hom}(T)\otimes A^1_{hom}(T)\ \twoheadrightarrow\ A^2_{AJ}(T)^\iota=A^2_{AJ}(S)\ .\]
%    
%  \noindent
%  (\rom3) The Abel--Jacobi map
%    \[  \Phi\colon\ \  \alb(T)\ \to\ \JJ(Y) \]
%    is an isogeny (where $\JJ(Y)$ is the intermediate Jacobian of $Y$), and 
%    \[ \Phi(T)=2\Theta^8/ 8!\ \ \ \hbox{in}\ H^{16}(\JJ(Y),\QQ)\ .\]
%  \end{theorem}
%  
%  \begin{proof} Point (\rom1) is proven in \cite{Wel}. Points (\rom2) and (\rom3) are \cite[Corollaire 3.2(b)]{V8} resp. \cite[Exemple 3.19(b)]{V8}.
%  \end{proof}
%  
%  \begin{remark} Let us assume the quartic $K3$ surface $Q$ does not contain a line. As explained in \cite{Fer} (cf. also \cite[example 3.5]{Beau1} and \cite[Remark 8.5]{Huy1}), the surface $S$ of theorem \ref{wel} is (isomorphic to) the so--called ``surface of bitangents'', which is the fixed locus of Beauville's anti--symplectic involution
%    \[ Q^{[2]}\ \to\ Q^{[2]} \]
%   first considered in \cite{Beau0}. As noted in \cite[Example 3.5]{Beau1}, this gives another proof of the fact that $p_g(S)=45$.
%   \end{remark}

\vskip1cm
\begin{nonumberingt} Thanks to the wonderful staff of the Executive Lounge at the Schilik Math Research Institute.
\end{nonumberingt}

\end{document}